\documentclass[reqno]{amsart}
\usepackage{amssymb,latexsym,amsmath,amsthm,enumerate}
\usepackage[mathscr]{eucal}

\theoremstyle{plain}
\newtheorem{theorem}{Theorem}
\newtheorem{thm}[theorem]{Theorem}%[section]
\newtheorem{lem}[theorem]{Lemma}

\newtheorem{remark}[theorem]{Remark}

\theoremstyle{definition}
\newtheorem{defn}[theorem]{Definition}

\newcommand{\Eq}{\operatorname{Eq}}
\DeclareMathOperator{\len}{len}
\DeclareMathOperator{\rel}{rel}
\DeclareMathOperator{\gen}{gen}
\DeclareMathOperator{\rellen}{rellen}

\newcommand{\up}[1]{\textup{#1}}

\newcommand{\tr}{\triangleright}

\newcommand{\bigand}{\operatornamewithlimits{\hbox{\Large$\&$}}}

\begin{document}

\title{Low growth equational complexity}

\author{Marcel Jackson}
\address{Department of Mathematics and Statistics, La Trobe University, Victoria  3086,
Australia \textup{(m.g.jackson@latrobe.edu.au)}} 

\subjclass[2010]{Primary 08B05; Secondary 20M07, 20D99}

\keywords{quasivariety, group, equational complexity, growth}

\begin{abstract}
The equational complexity function $\beta_\mathscr{V}:\mathbb{N}\to\mathbb{N}$ of an equational class of algebras $\mathscr{V}$ bounds the size of equation required to determine membership of $n$-element algebras in $\mathscr{V}$.  Known examples of finitely generated varieties $\mathscr{V}$ with unbounded equational complexity have growth in $\Omega(n^c)$, usually for $c\geq \frac{1}{2}$.  We show that much slower growth is possible, exhibiting $O(\log_2^3(n))$ growth amongst varieties of semilattice ordered inverse semigroups and additive idempotent semirings.   We also examine a quasivariety analogue of equational complexity, and show that a finite group has polylogarithmic quasi-equational complexity function, bounded if and only if all Sylow subgroups are abelian.
\end{abstract}
\thanks{The author was partially supported by ARC Discovery Project DP1094578 and Future Fellowship 120100666}

\maketitle

%%%%%%%%%%%%%%%
%%%%%%%%%%%%%%%
%%%%%%%%%%%%%%%
%%%%%%%%%%%%%%%
%%%%%%%%%%%%%%%
\section{Introduction}
In this article, an \emph{algebra} means a universal algebra, though our primary focus is on the class of finite groups and finite semilattice-ordered semigroups.  For a fixed signature $\mathscr{S}$ of operations, an \emph{equation} is an expression
${\bf u}\approx {\bf v}$, where ${\bf u}$ and ${\bf v}$ are terms in $\mathscr{S}$.  The equation is \emph{satisfied} in an algebra in the signature $\mathscr{S}$ if all interpretations $\theta$ of the variables into the universe of the algebra results in ${\bf u}\theta={\bf v}\theta$.  

The class of all $\mathscr{S}$-algebras satisfying some given system of equations is called the \emph{variety} defined by the equations.  A variety is always closed under homomorphisms ($\mathbb{H}$), isomorphic copies of subalgebras ($\mathbb{S}$) and direct products ($\mathbb{P}$), and conversely, every $\mathbb{H},\mathbb{S},\mathbb{P}$ closed class of similar algebras is a variety, definable by the equations holding true in all of its members; see Birkhoff~\cite{bir} or a text such as Burris and Sankappanavar \cite{bursan}.  We let~$\mathbb{V}(\mathscr{K})$ denote the the variety generated by a class $\mathscr{K}$, and write $\mathbb{V}({\bf A})$ to denote $\mathbb{V}(\{{\bf A}\})$ when ${\bf A}$ is a single algebra.

A challenging computational problem arises when one wishes to decide membership of a finite algebra ${\bf B}$ in a variety $\mathscr{V}$: even when $\mathscr{V}=\mathbb{V}({\bf A})$ for a finite algebra~${\bf A}$, this problem can be as hard as 2\texttt{EXPTIME}-complete (Kozik \cite{koz09}), and even amongst almost classical algebras, such as semigroups, there are examples for which the problem is \texttt{NP}-hard (Jackson and McKenzie~\cite{jacmck}, Jackson~\cite{jac:flexi}) and co-\texttt{NP}-complete (Kunc, Klima and Polak~\cite{KKP}).  For general varieties~$\mathscr{V}$---even recursively axiomatizable varieties---the problem can be undecidable (see Hirsch and Hodkinson~\cite{hirhod} for example).  Such membership problems are obviously fundamental in the general study of varieties, but they are also an important particular case of the more general situation of deciding membership of finite algebras in ``pseudovarieties'': classes of finite algebras closed under $\mathbb{H}$, $\mathbb{S}$ and taking finitary direct products.  A substantial motivation for this more general case is that many questions in formal languages can be recast in terms of membership problems for semigroup pseudovarieties.  See Almeida \cite{alm}, Eilenberg \cite{eil} or Rhodes and Steinberg~\cite{rhoste} for general theory, and Volkov \cite{vol} for some important concrete example cases where the pseudovariety is the finite part of $\mathbb{V}({\bf A})$ for some finite semigroup~${\bf A}$.

As every variety $\mathscr{V}$ has an equational characterization, a first attempt at deciding membership of a candidate algebra ${\bf B}$ in $\mathscr{V}$ might be to explore satisfaction of these characterizing equations.  Indeed, when $\mathscr{V}$ admits a characterization in terms of finitely many equations, then testing these equations provides a polynomial time algorithm to decide membership.  Even when there is no finite equational characterization for $\mathscr{V}$, the equational theory of $\mathscr{V}$ may still be well behaved enough for an approach in this style.  In particular, if ${\bf B}$ fails some identity of $\mathscr{V}$, it must fail some identity of~$\mathscr{V}$ involving at most as many variables as there are generators for~${\bf B}$.  When $\mathscr{V}$ is locally finite (and in particular, when $\mathscr{V}$ is finitely generated), there always exists a finite set of equations of $\mathscr{V}$ that capture the $n$-variable equational theory (for any $n$); this follows from Birkhoff's Finite Basis Theorem \cite[Theorem V.4.2]{bursan}.  When~$\mathscr{V}=\mathbb{V}({\bf A})$ for a finite algebra ${\bf A}$, it is possible to calculate a concrete bound on the size of the required $n$-variable equations, and thus obtain an in-principle algorithm for deciding membership.  The equational complexity function $\beta_\mathscr{V}$ captures exactly these notions.  

The \emph{equational complexity function} $\beta_\mathscr{V}\colon \omega\to\omega$ of $\mathscr{V}$ is defined by letting $\beta_\mathscr{V}(n)$ be the smallest number $\ell$ such that for every algebra ${\bf B}$ of size less than $n$, if ${\bf B}\notin\mathscr{V}$ then there is an equation of $\mathscr{V}$ failing on ${\bf B}$ and with length at most $\ell$.
Equivalently, $\beta_\mathscr{V}(n)$ is the smallest number $\ell$ such that for algebras of size less than $n$, lying in $\mathscr{V}$ is equivalent to satisfying the equations true in $\mathscr{V}$ of length at most $\ell$.  If $\mathscr{V}$ is the variety generated by a single finite algebra ${\bf A}$ then we also write $\beta_{\bf A}$ in place of $\beta_\mathscr{V}$.

Obviously, the definition of $\beta_\mathscr{V}$ depends on the precise definition of ``length''.  We follow McNulty, Szek\'ely and Willard \cite{MSzW} and define the length to be the number of symbols in the concatenation of the bracket-free prefix expressions for the two terms in the equation.  Thus the law $x\cdot (y\cdot y^{-1})\approx x$ becomes ${\cdot}x{\cdot}y{{}^{-1}}yx$ of size $7$ (noting that the expression ``$-1$'' is here a single operation symbol).

Existing work on equational complexity has had two main focal points.  The first is to identify the limits of fast growth of $\beta_{\bf A}$.  In Kun and V\'ertesi \cite{kunver} it is shown that $\Theta(n^k)$ growth is possible (for any $k\in\mathbb{N}$), while Kozik \cite{koz04} showed that at least exponential growth is possible.  McNulty, Szekely and Willard \cite{MSzW} give numerous concrete examples of algebras whose equational complexity is sandwiched somewhere between linear and quadratic growth, while Jackson and McNulty \cite{jacmcn} give a linear growth rate for the equational complexity of Lyndon's algebra.  

The second focus of existing work relates to a long-standing open problem due to Eilenberg and Sch\"utzenberger~\cite{eilsch}: \emph{if the pseudovariety of ${\bf A}$ can  be defined by finitely many equations, is it true that the variety generated by a finite algebra ${\bf A}$ can be defined by finitely many equations?}   It is a straightforward exercise to verify that (for finite signature), the function $\beta_\mathscr{V}$ is bounded above by a constant if and only if $\mathscr{V}$ can be defined, amongst finite algebras by a finite set of equations.  Thus the Eilenberg-Sch\"utzenberger problem is equivalent to asking: is it true that every finite algebra with bounded equational complexity has a finite basis for its equations?  These connections are explored in~\cite{MSzW} (though there is other work relating to the Eilenberg-Sch\"utzenberger conjecture that avoids discussion of equational complexity).

In Jackson and McNulty  \cite{jacmcn} it is suggested that as well as finding high growth equational complexity, there should be equal interest in finding slow but unbounded growth.  In particular,  algebras of very slow, but unbounded growth,  appear to be more likely related to difficult unresolved issues relating to axiomatizability.   No contributions have been made in this direction to date, with all examples known to the author at the time of writing growing at least at $\Omega(n^c)$, where $c\geq 1/2$ is typical. In the present article we show that a finite naturally semilattice-ordered Clifford semigroup always has equational complexity in $O(\log_2^3(n))$ (where $\log^k(n)$ denotes the polylog $(\log(n))^k$) and identify those with $\beta$-function in $O(1)$; see Theorem~\ref{thm:mainClifford}.  The smallest example with non-constant bounded (but slow) equational complexity has~$9$ elements.  To prove this we begin by introducing a corresponding theory of  \emph{quasi-equational complexity}, which refers to quasi-equations and quasivarieties in place of equations and varieties.  This concept also appears interesting enough in its own right, but is integral to our approach.  We show in Theorem \ref{thm:group} that every finite group has quasi-equational complexity  in $O(\log_2^3(n))$ and identify precisely those with growth in~$O(1)$.  (These ideas appear to have some relationship with the Short Presentation Conjecture of Babai, Goodman, Kantor, Luks and P\'alfy~\cite{bgklp}; see Remark~\ref{rem:short}.)  We then apply methods developed in \cite{jac:flat} to obtain  corresponding results for the equational complexity of naturally semilattice-ordered Clifford semigroups and examples from additive idempotent semirings.

\section{Equations, quasi-equations and quasi-equational complexity}
Throughout, we use lower case letters (with or without subscripts) for variables and boldface lower case letters (with or without subscripts) for generic terms and words built from variables.  Thus ${\bf u}$ will always denote  a word or term in an alphabet of variables, while $u$ will denote an individual variable.  We let $\Eq({\bf A})$ denote the set of all equations true on an algebra ${\bf A}$ (in some fixed, but otherwise arbitrary countably infinite set of variables).  The notation $\Eq_n({\bf A})$ will denote the subset of $\Eq({\bf A})$ consisting of those equations of size at most $n$.

A \emph{quasi-equation} is an expression of the form
\[
\bigg(\bigand_{1\leq i\leq n}{\bf u}_i\approx {\bf v}_i\bigg) \rightarrow {\bf u}\approx {\bf v}
\]
for some $n\geq 0$, where each of the ${\bf u}_i$ and ${\bf v}_i$ are terms, $\bigand$ is logical conjunction and $\rightarrow$ is logical implication.  A quasi-equation is satisfied by an algebra if every interpretation of the variables that leads to the premise of the implication being true also leads to the conclusion being true.  Equations correspond to the case of $n=0$. 
 The class of algebras defined by a system of quasi-equations is known as a \emph{quasivariety} and is closed under~$\mathbb{S}$,~$\mathbb{P}$ and ultraproducts $\mathbb{P}_{\rm u}$.  Every $\mathbb{S},\mathbb{P},\mathbb{P}_{\rm u}$ closed class---an indeed, any $\mathbb{SPP}_{\rm u}$-closed class---is a quasivariety, and when ${\bf A}$ is a finite algebra, then the quasivariety generated by ${\bf A}$ can be written as $\mathbb{SP}({\bf A})$, without ultraproducts.  It is obvious from the syntactic definitions and also the semantic equivalent conditions, that every variety is a quasivariety.  We direct the reader to~\cite[Chapter~5]{bursan} or Gorbunov \cite{gor} for a treatment of quasivarieties and quasi-equations.

Equational complexity generalises to quasi-equational complexity in an obvious way: up to big-O equivalence, again by concatenating all terms in prefix notation.  Thus the quasi-equation $x\cdot y\approx y\cdot x\rightarrow x\approx y$ becomes ${\cdot}xy{\cdot} yxxy$ of length $8$.  
We let $\overline{\beta}_\mathscr{Q}$ denote the quasi-equational complexity function of the quasivariety $\mathscr{Q}$, also writing $\overline{\beta}_{\bf A}$ to denote~$\overline{\beta}_{\mathbb{SPP}_{\rm u}({\bf A})}$.  
For a fixed variety $\mathscr{V}$ (hence also a quasivariety), the relationship $\overline{\beta}_\mathscr{V}(n)\leq \beta_\mathscr{V}(n)$ for all~$n$ follows immediately from the definitions, but in contrast $\beta_{\bf A}$ and $\overline{\beta}_{\bf A}$ may be quite different because the classes $\mathbb{SPP}_{\rm u}({\bf A})$ and~$\mathbb{HSP}({\bf A})$ are, in general, quite different.

\section{Finite groups with polylog quasi-equational complexity.}\label{sec:group}
Ol'shanski\u{\i} \cite{ols} has shown that the quasivariety $\mathbb{SP}({\bf G})$ of a finite group ${\bf G}$ has a finite axiomatization by quasi-equations if and only if all Sylow subgroups of ${\bf G}$ are abelian.  When~${\bf G}$ has a finite basis for its quasi-equations, then the quasi-equational complexity of~${\bf G}$ is bounded by a constant.  When~${\bf G}$ has no finite basis for its quasi-equations (that is, when it contains a nonabelian Sylow subgroup), then %it follows from Ol'shanski\u{\i}'s argument that 
for all $n$ there is a finite group~${\bf H}$ such that every $n$-generated subgroup of ${\bf H}$ lies in the quasivariety of ${\bf G}$, but that ${\bf H}$ itself does not.  This is shown by Ol'shanski\u{\i} in the discussion following the proof of \cite[Lemma~2]{ols} (where ${\bf H}$ is denoted by $C$).    It follows that $\overline{\beta}_{\bf G}(|H|)>\overline{\beta}_{\bf G}(n)$, so that $\overline{\beta}_{\bf G}$ grows unbounded.  The rest of this section is devoted to showing that this growth can be contained within the class $O(\log_2^3(n))$, thus proving the following theorem.
\begin{thm}\label{thm:group}
Let ${\bf G}$ be a finite group.  If all Sylow subgroups of ${\bf G}$ are abelian, then the quasi-equational complexity function $\overline{\beta}_{\bf G}$ is bounded by a constant.  If ${\bf G}$ contains a nonabelian Sylow subgroup, then $\overline{\beta}_{\bf G}$ is unbounded, but lies in $O(\log_2^3(n))$.
This is true for all subsignatures of $\{\cdot,{}^{-1},1\}$ containing $\cdot$.
\end{thm}

We leave discussion of the final sentence to the end of this section (see Remark~\ref{rem:signature}).  Observe that to complete the proof of the other statements in Theorem \ref{thm:group} it only remains to show that $\overline{\beta}_{\bf G}$ can  be found in $O(\log_2^3(n))$ (for any finite group ${\bf G}$; we will not need to use the fact that ${\bf G}$ contains a nonabelian subgroup).  To prove this we begin by showing that for any fixed finite group ${\bf G}$, and ${\bf H}\in\mathbb{V}({\bf G})$, then if ${\bf H}\notin\mathbb{SP}({\bf G})$ there is a quasi-equation of size $O(\log_2^3(|H|))$ satisfied by ${\bf G}$ and failing on ${\bf H}$.  In particular, if ${\bf G}$ has a nonabelian Sylow subgroup---whence generates a nonfinitely axiomatizable quasivariety---this shows that membership of finite groups ${\bf H}$ in the quasivariety $\mathbb{SP}({\bf G})$ can be verified by testing quasi-equations of size only up to $O(\log_2^3(|H|))$.  We use this to show that the flat extension of the finite group ${\bf G}$ exhibits equational complexity bounded by $O(\log_2^3(|H|))$, but (when ${\bf G}$ has a nonabelian Sylow subgroup) not by any constant.

\begin{remark}\label{rem:short}
The first half of the argument is closely related to the results of Babai, Goodman, Kantor, Luks and P\'alfy \cite{bgklp} who show that it is possible to associate with each finite group ${\bf H}$ a presentation of short size.  The authors of \cite{bgklp} state a \emph{Short Presentation Conjecture}: that the presentation of ${\bf H}$ can be made to be of total size $O(\log_2^3(|H|))$.  
\end{remark}

The Short Presentation Conjecture is not resolved here, but our approach is at least slightly reminiscent of the methods invoked in \cite{bgklp}.  We consider an arbitrary finite group~${\bf H}$ in the variety of a fixed finite group ${\bf G}$ and examine a composition series for ${\bf H}$, lifting presentations for the various simple groups to one for ${\bf H}$.  The precise description is quite different to that in \cite{bgklp}.  Also, because the simple groups arising from the composition series also lie in $\mathbb{V}({\bf G})$, we have access to the following fact. (This folklore lemma was pointed out to the author by Mikhail Volkov.)
\begin{lem}\label{lem:sg}
Let ${\bf G}$ be a finite group.  Up to isomorphism, there are only finitely many finite simple groups in the variety $\mathbb{V}({\bf G})$.
\end{lem}
\begin{proof}
Claim 51.22 of H.~Neumann \cite{neu} shows that all simple groups in $\mathbb{V}({\bf G})$ have order at most $|G|$.\end{proof}
   We examine composition series to lift presentations for these to short presentations to arbitrary ${\bf H}\in \mathbb{V}({\bf G})$. This approach is common to Babai et al.~though the precise construction of the presentation is quite different.  
   In a further deviation from Babai et al.~\cite{bgklp}, we will also require the additional property that every element of ${\bf H}$ can be written as a short product of generators: $O(\log_2^3(|H|))$ would suffice, but $O(\log_2(|H|))$ is shown.  

In the following lemma we consider group presentations in the variety of groups, using the convention that relators are single words.  Thus, if ${\bf w}$ is a relator, we mean that ${\bf w}=1$ in the group.
\begin{lem}\label{lem:lift}
\up(The lifting lemma.\up)
Let ${\bf N}\trianglelefteq {\bf M}$ be finite groups and assume that 
\begin{itemize}
\item 
${\bf N}$ has presentation $\langle a_1,\dots,a_k\mid \{{\bf w}_i:i=1,\dots, k'\}\rangle$, and
\item 
${\bf M}/{\bf N}$ has presentation $\langle Nb_1,\dots,Nb_\ell\mid \{N{\bf v}_i:i=1,\dots, \ell'\}\rangle$,
\end{itemize}
where ${\bf w}_i$ are some words  in the alphabet $\{a_1,\dots,a_k\}$, and ${\bf v}_i$ are some words in the 
alphabet $\{b_1,\dots,b_\ell\}$.  
Then ${\bf M}$ is generated by the set $\{a_1,\dots,a_k,b_1,\dots,b_\ell\}$ and can be presented by the following words\up:
\begin{enumerate}
\item ${\bf w}_i$ for $i=1,\dots,k'$\up;
\item \up(for each $i=1,\dots,k$ and $j=1,\dots,\ell$\up) ${\bf w}_{i,j}b_ia_j^{-1}b_i^{-1}$ where ${\bf w}_{i,j}$ is a word in $\{a_1,\dots,a_k\}$ of minimal length such that ${\bf w}_{i,j}b_i=b_ia_j$ \up(the word ${\bf w}_{i,j}$ exists as $b_ia_j\in b_i{\bf N}={\bf N}b_i$\up)\up;
\item \up(for each $i=1,\dots,\ell'$\up)  ${\bf u}_i^{-1}{\bf v}_i$, where ${\bf u}_i$ is a word of minimal length in the alphabet $\{a_1,\dots,a_k\}$ for which ${\bf v}_i={\bf u}_i$.
\end{enumerate}
Moreover, every element of ${\bf M}$ can be written in the form ${\bf w}_a{\bf w}_b$, where ${\bf w}_a$ is a word in the generators $a_1,\dots,a_k$ and ${\bf w}_b$ is a word in the generators $b_1,\dots,b_\ell$.
\end{lem}
\begin{proof}
First note that the presentation is well-defined and that each of the given words does equal $1$ in ${\bf M}$. This is immediate for Item (1).  For Item (2), as $b_iN=Nb_i$, and $a_j\in N$, we have that $b_ia_j\in Nb_i$, showing the existence of ${\bf w}_{i,j}\in N$ with $b_ia_j={\bf w}_{i,j}b_i$ in~${\bf M}$; note then that ${\bf w}_{i,j}b_ia_j^{-1}b_i^{-1}=1$.  For item (3), observe that as $N{\bf v}_i$ is in the presentation of ${\bf M}/{\bf N}$ we have that ${\bf v}_i\in N$, so that there exists a word ${\bf u}_i$ in the generators $\{a_1,\dots,a_k\}$ with ${\bf v}_i={\bf u}_i$ and hence ${\bf u}_i^{-1}{\bf v}_i=1$.  Thus each of the relator words is equal to $1$ in ${\bf M}$ as required.  Also, every element $h$ of ${\bf M}$ lies in a coset of the form ${\bf N}{\bf w}_b$ of ${\bf M}$, where ${\bf w}_b$ is a product of the elements $b_i$.  Thus there is an element ${\bf w}_a\in {\bf N}$ such that $h={\bf w}_a{\bf w}_b$.  So $\{a_1,\dots,a_k,b_1,\dots,b_\ell\}$ is a set of generators of ${\bf M}$.  It remains to show that if ${\bf w}$ is a word in $\{a_1,\dots,a_k,b_1,\dots,b_\ell\}$ equal to $1$ in ${\bf M}$, then this fact can be deduced using the given list of relators. 

 Assume that ${\bf w}$ is a product in the generators that is equal to $1$.  First observe that the laws of the second kind ensure that ${\bf w}$ can be written in the form ${\bf v}_a{\bf v}_b$, where ${\bf v}_a$ is a product of the elements $a_1,\dots,a_k$ while ${\bf v}_b$ is a product of the elements $b_1,\dots,b_\ell$.  Now as ${\bf w}=1$ we have ${\bf N}{\bf v}_b={\bf N}$.  This may be achieved by use of the laws of the third kind: invoking laws of the second kind freely to move all generated occurrences of generators in $a_1,\dots,a_k$ to the left of any from $b_1,\dots,b_\ell$.  Thus eventually we obtain a deduction of ${\bf v}_a{\bf v}_b={\bf u}_a$, where ${\bf u}_a$ is a product of $a_1,\dots,a_k$ that is equal to $1$ in ${\bf M}$, and hence in ${\bf N}$.  This can be established using laws of the first kind.
\end{proof}

\begin{lem}\label{lem:log3bound}
There is a constant $c$ such that every  finite group ${\bf H}\in \mathbb{V}({\bf G})$ has a presentation $\langle C;R\rangle$ of total length at most $c\log_2^3(|H|)$ and such that every element in $H$ can be written as a product of length at most $2|G|\log_2(|H|)$.
\end{lem}
\begin{proof}
By Lemma \ref{lem:sg} there are only finitely many finite simple groups in the variety generated by ${\bf G}$, say $\{{\bf S}_1,\dots,{\bf S}_p\}$.  Thus there is a constant bound on the size of a presentation for the ${\bf S}_i$.
  Now consider the composition series $\{e\}={\bf H}_0\trianglelefteq {\bf H}_1
\trianglelefteq {\bf H}_2\trianglelefteq \dots\trianglelefteq {\bf H}_m={\bf H}$.  Observe that $m\leq\log_2(|H|)$, because $|H_{i}|\leq \frac{|H_{i+1}|}{2}$ and that each ${\bf H}_{i+1}/{\bf H}_i\in \mathbb{I}\{{\bf S}_1,\dots,{\bf S}_p\}$ (the isomorphism closure of $\{{\bf S}_1,\dots,{\bf S}_p\}$).  Now Lemma \ref{lem:lift} enables the inductive construction of a presentation for ${\bf H}$: at the $(i+1)$th step, the group ${\bf H}_{i+1}$ plays the role of ${\bf M}$ in Lemma \ref{lem:lift}, and the group ${\bf H}_i$ plays the role of ${\bf N}$.  At completion, a set of generators 
\[
a_{1,1},\dots,a_{1,n_1},\dots,a_{m,1},\dots,a_{m,n_m}
\] 
has been constructed for ${\bf H}$, where at the $(i+1)$th step of the induction, the elements $a_{1,1},\dots,a_{i,n_i}$ played the role of the elements $a_1,\dots,a_k$ in the lemma, while $a_{i+1,1},\dots,a_{i+1,n_{i+1}}$ played the role of the $b_1,\dots,b_\ell$.
\begin{enumerate}
\item Let $\#\gen$ denote the largest generating set for the fixed selection of presentations for $\{{\bf S}_1,\dots,{\bf S}_p\}$.  Inductively let $\gen(1)$ denote the number of generators for ${\bf H}_1\in \{{\bf S}_1,\dots,{\bf S}_p\}$ (so $\gen(1)=n_1$) and $\gen(i+1)$ be the number of generators for ${\bf H}_{i+1}$ as constructed from ${\bf H}_i$ and the quotient ${\bf H}_{i+1}/{\bf H}_i\in \{{\bf S}_1,\dots,{\bf S}_p\}$.  It is routine to see that $\gen(i)=\sum_{j=1}^in_j\leq i\cdot \#\gen$.  We now count the size of this presentation, and the worst case upper bound on the length of a product of generators needed to express elements of $H$.
\item Let $\#\len$ denote the worst case minimal length of any product of generators required to represent elements of $\{{\bf S}_1,\dots,{\bf S}_p\}$.  Inductively, let $\len(1)$ be the worst case minimal length of a word required to represent elements of ${\bf H}_1$, and let $\len(i+1)$ denote the shortest length of a product of generators required to represent elements in ${\bf H}_{i+1}$ as constructed using Lemma \ref{lem:lift}.  Note that by the final statement of Lemma \ref{lem:lift} we have $\len(i+1)\leq \#\len+\len(i)$ so that $\len(i)\leq i\cdot \#\len$.
\item Let $\#\rel$ denote the largest number of relations used in the fixed selection of presentations for $\{{\bf S}_1,\dots,{\bf S}_p\}$. Inductively let $\rel(1)$ denote the number of relations for ${\bf H}_1\in \{{\bf S}_1,\dots,{\bf S}_p\}$ and $\rel(i+1)$ be the number of relations constructed for ${\bf H}_{i+1}$ using Lemma \ref{lem:lift} from ${\bf H}_i$ and the quotient ${\bf H}_{i+1}/{\bf H}_i\in \{{\bf S}_1,\dots,{\bf S}_p\}$.  Note that $\rel(i+1)\leq \rel(i)+\#\gen\cdot \gen(i)+\#\rel$.
\item $\#\rellen$ denotes the maximal length of any relation appearing in the fixed presentations for $\{{\bf S}_1,\dots,{\bf S}_p\}$.  Inductively, $\rellen(1)$ denotes the maximal length of any relation in the presentation for ${\bf H}_1$, and $\rellen(i+1)$ denotes the maximal length of any relation in the presentation constructed for ${\bf H}_{i+1}$.  Note that 
\[
\rellen(i+1)\leq\max\{
\stackrel{(1)}{\overbrace{\rellen(i)}},\stackrel{(2)}{\overbrace{3+\len(i)}},\stackrel{(3)}{\overbrace{\#\rellen+\len(i)}}\}
\]
\end{enumerate}
where the numbering states which case of Lemma \ref{lem:lift} the expression derives from.
We have a system of $4$ simultaneous recurrence relations, however we have already observed easy bounds for $\gen(i)$ and $\len(i)$.  These can be substituted into the recurrence relations for $\rel(i)$ and  $\rellen(i)$.  For $\rel(i)$ we have, after substituting the bound $\#\gen \cdot i$ for $\gen(i)$:
\[
\rel(i+1)\leq \rel(i)+i\cdot \#\gen^2+\, {\#\rel}  \leq i(i + 1)/2 \cdot \#\gen^2 +\, (i+1) \cdot \#\rel
\]
Thus $\rel(m)\in O(\log_2^2(|H|))$.

Finally, for $\rellen(i)$ we observed 
\[
\rellen(i+1)\leq\max\{\rellen(i),3+\len(i),\#\len+\len(i)\}.
\]  
Using $\len(i)\leq i\cdot \#\len$ we obtain a linear bound on the growth of $\rellen(i)$, so that $\#\rellen(m)\in O(\log_2(|H|))$.

Combining all this, at completion, we have obtained a presentation for ${\bf H}={\bf H}_m$ in $O(\log_2(|H|))$ generators, with at most $O(\log_2^2|H|)$ relations, each of maximal length $O(\log_2|H|)$.  Thus the total length of the presentation is in $O(\log_2^3(|H|))$.  Also, every element of ${\bf H}$ can be written as a $O(\log_2(|H|))$ length product of the constructed generating set.
\end{proof}

\begin{lem}\label{lem:shortqe}
For a finite group ${\bf G}$, if ${\bf H}$ is a finite group and ${\bf H}\notin \mathbb{Q}({\bf G})$, then there is a quasi-equation $\phi$ satisfied by ${\bf G}$ and failing on ${\bf H}$ and the total length of $\phi$ is $O(\log_2^3(|H|))$.
\end{lem}
\begin{proof}
By the theorem of Oates-Powell  \cite{oatpow}, there is a finite basis for the variety generated by ${\bf G}$.  Let us fix such a basis, and denote it by $\Sigma$.  Observe that there is a constant bound on the size of equations in $\Sigma$ (because it's finite) so if ${\bf H}\notin\mathbb{V}({\bf G})$ then ${\bf H}$ fails an identity of $\Sigma$, and we are done.  Now assume that ${\bf H}\in\mathbb{V}({\bf G})$.  
As ${\bf H}\notin\mathbb{Q}({\bf G})$ there is an element $h\in H\backslash\{1\}$ such that every homomorphism from~${\bf H}$ into ${\bf G}$ identifies $h$ with~$1$.  By Lemma \ref{lem:log3bound}, we may select a presentation  $\langle C;R\rangle$ for ${\bf H}$ of total length $O(\log_2^3(|H|))$ and assume that $h$ is written as a product ${\bf w}$ of generators of length at most $O(\log_2(|H|))$.   We now consider the presentation for~${\bf H}$ as the premise of a quasi-identity.  Let $\phi$ be the quasi-identity
\begin{equation*}
\left(\bigand_{{\bf u}\in R}{\bf u}\approx 1\right)\rightarrow {\bf w}\approx 1.\tag{$\dagger$}\label{eq:dagger}
\end{equation*}
Obviously ${\bf H}$ fails \eqref{eq:dagger}, however ${\bf G}\models \phi$ because any evaluation into ${\bf G}$ that satisfies the premise of \eqref{eq:dagger} yields a homomorphism from ${\bf H}$ into ${\bf G}$, and all such homomorphisms identify ${\bf w}$ and $1$.  The total length of $\phi$ is $O(\log_2^3(|H|))$, as required.
\end{proof}
\begin{remark}\label{rem:signature}
Lemma \ref{lem:shortqe} remains true in the signatures $\{\cdot\}$, $\{\cdot,1\}$ and $\{\cdot,{}^{-1}\}$.
\end{remark}
\begin{proof}
We consider the case of $\{\cdot\}$, with the other cases following by using a subset of the argument.  We are given a presentation $\langle A\mid R\rangle$ for a group ${\bf G}$, where $R$ is a set of group words in the alphabet $A$.  To remove ${}^{-1}$ from the signature, first let~$A^{-1}$ denote the set $\{a^{-1}\mid a\in A\}$, which we now treat as an alphabet disjoint to~$A$.  Observe that the law $(xy)^{-1}\approx y^{-1}x^{-1}$ allows us to assume that each ${\bf w}\in R$ is a semigroup word in the alphabet $A\cup A^{-1}$.  Then $\langle A\cup A^{-1}\mid R\cup\{aa^{-1}\mid a\in A\}\rangle$ is a monoid presentation for ${\bf G}$ (in the signature $\{\cdot,1\}$).  To remove $1$ from the signature, add it as a generator and add the relators $g\cdot 1=g$ and $1\cdot g=g$ for each generator $g\in A\cup A^{-1}\cup\{1\}$ to $R$.  It is easy to see that the total length of the presentation is extended by only a constant factor, while the shortest length of a product to represent an element remains unchanged.  The quasi-equation in Equation \eqref{eq:dagger} is amended accordingly, with ``${\bf u}\in R$'' replaced by ``${\bf u}\approx 1\in R$''.
\end{proof}
Lemma \ref{lem:shortqe} completes the proof of Theorem \ref{thm:group}, except for the statements about signature.  Remark \ref{rem:signature} shows that the theorem holds in each of the described signatures.

\section{Polylog equational complexity for a semilattice ordered inverse semigroup}
Recall that an inverse semigroup is an involuted semigroup satisfying $xx^{-1}yy^{-1}\approx yy^{-1}xx^{-1}$; see a text such as Howie \cite{how}, Lawson \cite{law} or Petrich \cite{pet}.  A \emph{naturally semilattice-ordered inverse semigroup} is an inverse semigroup with a second binary operation $\wedge$ satisfying the usual semilattice axioms, along with left and right distributivity of $\cdot$ over $\wedge$, and the laws
\[
x\wedge y \approx x(x\wedge y)^{-1}(x\wedge y)\text{ and }(x\wedge y)zz^{-1}=(xzz^{-1})\wedge (yzz^{-1})
\]
which ties the usual semilattice order defined by $\wedge$ to the usual inverse semigroup theoretic order defined by $\cdot,{}^{-1}$.  We will make particular use of the property 
\[
(x\wedge y)zz^{-1}=x\wedge (yzz^{-1}),\tag{$\ddagger$}\label{eqndd}
\]
which is an easy consequence of the above laws.
An inverse semigroup is said to be a \emph{Clifford semigroup} if the idempotent elements are central: $xx^{-1}y\approx yxx^{-1}$.  

Naturally semilattice-ordered Clifford semigroups can be found in the work of Leech~\cite{lee1,lee2} as well as in the study of algebras of injective partial maps, such as in Jackson and Stokes \cite{jacsto:06}.  

The main result of this section and of the article is the following theorem.
\begin{thm}\label{thm:mainClifford}
Let ${\bf C}$ be a finite naturally semilattice-ordered Clifford semigroup.  If all subgroups of ${\bf C}$ have only abelian Sylow subgroups, then ${\bf C}$ has a finite basis for its equational theory and $\beta_{\bf C}$ is bounded by a constant.  Otherwise, $\beta_{\bf C}$ is unbounded and in $O(\log_2^3(n))$.
\end{thm}
We also observe in Remark \ref{rem:semiring} that one may similarly obtain examples of finite additive idempotent semirings with equational complexity growing within $O(\log_2^3(n))$.

The remainder of the section concerns the relevant definitions and proofs to arrive at Theorem \ref{thm:mainClifford}.  We now employ a method developed in \cite{jac:flat} for translating quasi-equations of partial algebras into equations of some other kind of algebraic structure.  When applied to groups (as partial algebras that happen to be total), one arrives at the class of semilattice-ordered Clifford semigroups; this tight relationship is developed in Sections 7.7 and 7.8 of \cite{jac:flat}.

\begin{defn}
Let ${\bf G}=\langle G;\cdot,{}^{-1}\rangle$ be a group and $0$ be a symbol not appearing in $G$.  Define the \emph{flat extension} of $\flat({\bf G})$ to be the algebra on the set $G\cup \{0\}$ with operations $\cdot,{}^{-1},\wedge$, with $\cdot$ and ${}^{-1}$ extended by letting $0$ be an absorbing element and by 
\[
x\wedge y=\begin{cases}
x&\text{if $x=y$}\\
0&\text{otherwise}.
\end{cases}
\] 
We also allow the same notation for other signatures, in particular, for $\{\cdot\}$.
\end{defn}
The flat extension of a group is a naturally semilattice-ordered Clifford semigroup and is subdirectly irreducible.
It follows from Theorems 5.3 and 7.5 of \cite{jac:flat} that the 
class of subdirectly irreducible naturally semilattice-ordered Clifford 
semigroups is precisely the isomorphism closure of the class of flat extensions of groups.

The following lemma is essentially a trivial consequence of Sections 7.7 and 7.8 of \cite{jac:flat}, along with the fact a quasivariety generated by a finite set of finite groups is equal to the quasivariety generated by a single finite group.  We sketch a proof for completeness.
\begin{lem}\label{lem:singlegroup}
The variety generated by a finite naturally semilattice-ordered Clifford semigroup ${\bf C}$ is equal to one generated by the flat extension of a finite group ${\bf G}$, with ${\bf G}$ obtained as the direct product of a family of subgroups of the $\{\cdot\}$-reduct of ${\bf C}$.
\end{lem}
\begin{proof}
Let ${\bf C}$ be a finite naturally semilattice-ordered Clifford semigroup, and let ${\bf C}_1,\dots,{\bf C}_n$ be a complete list of its subdirectly irreducible quotients.  By elementary universal algebraic considerations, the variety $\mathscr{V}$ generated by ${\bf S}$ is equal to the variety generated by $\{{\bf C}_1,\dots,{\bf C}_n\}$.  By \cite[Theorems 7.5 and 5.3]{jac:flat}, each of the ${\bf C}_i$ is isomorphic to the flat extension of some group ${\bf G}_i$.  Each  $\flat({\bf G}_i)$ is a homomorphic image of ${\bf C}$, but by taking $\wedge$-minimal elements of ${\bf C}$ from each kernel class of this homomorphism (which exist as ${\bf C}$ is finite), we will show that each ${\bf G}_i$ is in fact isomorphic to a subgroup of ${\bf C}$.  

Let $\phi$ denote a surjective homomorphism from ${\bf C}$ onto $\flat({\bf G}_i)$, and for each $g\in G_i$, let $\overline{g}$ denote a minimal element of $\phi^{-1}(g)$.   As $g\wedge g=g$ in $\flat({\bf G}_i)$, the set $\phi^{-1}(g)$ is a $\{\wedge\}$-subalgebra of ${\bf C}$ so that $\overline{g}$ is the unique minimum of $\phi^{-1}(g)$.  Let $e$ denote the identity element of ${\bf G}_i$.  Then as $e=ee^{-1}$ we have that $\phi^{-1}(e)$ contains an idempotent $\overline{e}\cdot\overline{e}^{-1}$.  As $\overline{e}\leq \overline{e}\cdot\overline{e}^{-1}$ and ${\bf C}$ is an inverse semigroup, it follows that $\overline{e}$ is itself an idempotent.  Next, for any $g\in G$ we have $\overline{g}\cdot \overline{e}\geq \overline{g}$ by the minimality of $\overline{g}$ in $\phi^{-1}(g)$, giving $\overline{g}=\overline{g}\wedge(\overline{g}\cdot \overline{e})=(\overline{g}\wedge \overline{g})\cdot\overline{e}=\overline{g}\cdot\overline{e}$, by law~\eqref{eqndd}.  In particular then, we have $\overline{g^{-1}}\cdot\overline{g}\cdot\overline{e}=\overline{g^{-1}}\cdot\overline{g}$.  However $\overline{g^{-1}}\cdot\overline{g}\geq \overline{g^{-1}g}=\overline{e}$ by the minimality of~$\overline{e}$.  Because $\leq$ is the inverse semigroup order this implies $\overline{g^{-1}}\cdot\overline{g}\cdot\overline{e}=\overline{e}$.  Thus we have $\overline{g^{-1}}\cdot\overline{g}=\overline{e}$ for any $g\in G$. 

Finally then, for any two elements $g,h\in G$ we have  $\overline{g}\cdot \overline{h}\geq \overline{gh}$, while $\overline{gh}\cdot \overline{h^{-1}}\geq \overline{g}$ by minimality.  Then $\overline{gh}= \overline{gh}\cdot \overline{e}=\overline{gh}\cdot\overline{h^{-1}}\cdot \overline{h} \geq \overline{g}\cdot \overline{h} \geq \overline{gh}$ giving equality throughout.  Hence the $\wedge$-minimal elements of $\phi^{-1}(G_i)$ form a $\{\cdot\}$-subalgebra which is obviously isomorphic to~${\bf G}_i$.  

Let ${\bf G}$ be the direct product $\prod_{i}{\bf G}_i$.  The quasivariety generated by ${\bf G}$ is equal to that generated by $\{{\bf G}_1,\dots,{\bf G}_n\}$.  Hence by~\cite[Theorem~5.3]{jac:flat}, the variety $\mathscr{V}$ can equivalently be generated by $\flat({\bf G})$.
\end{proof}
In the case of a finite group (or of a finite naturally semilattice-ordered semigroup), it is not necessary that the operation ${}^{-1}$ be included, as it is a term function in multiplication: by $x^{-1}=x^{d-1}$, where $d$ is the period.  Notice that $x^{d}y$ is a term function acting as a second projection.
The following lemma summarises some of the key facts regarding varieties generated by the flat extensions, in the case where there is a second-projection term.  It is part of Theorems~5.3 and~5.12 of~\cite{jac:flat}.
\begin{lem}\label{lem:translate}\cite{jac:flat}
Let $\mathscr{Q}$ be a quasivariety of algebraic structures \up(of some fixed finite type\up) on which there is a two-variable term $x\tr y$ acting as second projection\up: $\mathscr{Q}\models x\tr y\approx y$.  Then the class of subdirectly irreducible members of $\mathbb{V}(\{\flat({\bf H})\mid{\bf H}\in \mathscr{Q}\})$ is $\mathbb{I}(\{\flat({\bf H})\mid{\bf H}\in \mathscr{Q}\})$.  Moreover, $\mathscr{Q}$ has a finite axiomatization by quasi-equations if and only if $\mathbb{V}(\{\flat({\bf H})\mid{\bf H}\in \mathscr{Q}\})$ has a finite axiomatization by equations.
\end{lem}
The last sentence in this lemma is proved in \cite{jac:flat} using \cite[Lemma 5.9]{jac:flat}, which gives an explicit translation of quasi-equations in the language of $\mathscr{Q}$ to equations in the language of $\{\flat({\bf H})\mid{\bf H}\in \mathscr{Q}\}$.  The translation involves only a linear adjustment in length when $\tr$ is a fundamental operation, however in the present setting we have~$\tr$ only as  term function, which can potentially result in an exponential increase in the length of the expression.  The issue is easily resolved: the following is a slight recasting of a particular case of \cite[Lemma 5.9]{jac:flat}, that avoids the exponential blowout.

\begin{lem}\label{lem:shorteq}
Let $\bigand_{1\leq i\leq n}({\bf u}_i\approx {\bf v}_i)\rightarrow {\bf u}_0\approx {\bf v}_0$ be a quasi-equation in the language of a single binary operation and with the property that every variable appearing in the expression appears somewhere in the premise of the implication.  Then the quasi-equation 
\[
\rho := \bigand_{1\leq i\leq n}({\bf u}_i\approx {\bf v}_i)\rightarrow {\bf u}_0\approx {\bf v}_0
\] 
holds on a group ${\bf H}$ of exponent $d$ if and only if the following equation holds on~$\flat({\bf H})$\up:
\begin{equation*}
\rho^\flat:=\bigg(\prod_{1\leq i\leq n}({\bf u}_i\wedge {\bf v}_i)\bigg)^d ({\bf u}_0\wedge {\bf v}_0)^d\approx \bigg(\prod_{1\leq i\leq n}({\bf u}_i\wedge {\bf v}_i)\bigg)^d.\tag{$*$}\label{equation}
\end{equation*}
\end{lem}
\begin{proof}
Let $0$ denote the bottom element with respect to $\wedge$ in $\flat({\bf H})$.  If $\rho$ fails in~${\bf H}$ under some substitution $\theta$ of the variables into $H$, then considering $\theta$ as an substitution into the flat extension $\flat({\bf H})$, we have the right hand side of $\rho^\flat$ in \eqref{equation} taking the value $\theta((\prod_{1\leq i\leq 1}({\bf u}_i\wedge {\bf v}_i))^d)=1$, while as $\theta({\bf u}_0)\neq \theta({\bf v}_0)$ the left hand side involves $\theta({\bf u}_0\wedge {\bf v}_0)=\theta({\bf u}_0)\wedge \theta({\bf v}_0)=0$.  So $\rho^\flat$  fails on $\flat({\bf H})$. 

Now assume that $\rho$ holds on ${\bf H}$.  It is clear that if the right hand side of $\rho^\flat$ 
takes the value $0$ under some interpretation of the variables in $\flat({\bf H})$, then so does the left hand side.  Let us assume then that $\prod_{1\leq i\leq 1}({\bf u}_i\wedge {\bf v}_i)$ does not take the value $0$ under some interpretation $\theta$ of the variables of the equation into $\flat({\bf H})$.  Note that all variables in $\rho^\flat$ appear in $\prod_{1\leq i\leq 1}({\bf u}_i\wedge {\bf v}_i)$, so in fact $\theta$ is an interpretation into ${\bf H}$.  Also, as $x\wedge y=0$ unless $x=y$, we must have that $\theta({\bf u}_i)=\theta({\bf v}_i)$ in ${\bf H}$ (for every $i=1,\ldots,n$).  As ${\bf H}\models \bigand_{1\leq i\leq n}({\bf u}_i\approx {\bf v}_i)\rightarrow {\bf u}_0\approx {\bf v}_0$ it follows that $\theta({\bf u}_0)=\theta({\bf v}_0)$ also, from which it is easily seen that both the left hand side and the right hand side of $\rho^\flat$ take the value $1$.
\end{proof}
Note that if ${\bf G}$ is fixed and ${\bf H}\in\mathbb{V}({\bf G})$, then the exponent of ${\bf H}$ is bounded by that of ${\bf G}$, hence the quasi-equation found in Lemma \ref{lem:shortqe} translates via Lemma \ref{lem:shorteq} to an equation of size $O\big(\log_2^3(|H|)\big)$.

Let ${\bf A}$ be a finite algebra.  A \emph{jump point} of 
$\beta_{\bf A}$ is a number $n\in \mathbb{N}$ such that $\beta_{\bf A}(n-1)<\beta_{\bf A}(n)$.  
A \emph{$\beta_{\bf A}$-critical} algebra is a finite
algebra ${\bf B}$ such that 
\begin{itemize}
    \item $|B|$ is a jump point for $\beta_{\bf A}$;
 \item ${\bf B}$ satisfies all identities of ${\bf A}$ up to length 
$\beta_{\bf A}(|B|)-1$
\item ${\bf B}$ fails some identity of ${\bf A}$ with length 
$\beta_{\bf A}(|B|)$.  
\end{itemize}
In other words, $\beta$-critical algebras are the algebras that 
force an increase in $\beta_{\bf A}$.

\begin{lem}\label{lem:bcritical}
    Every $\beta_{\bf A}$-critical algebra is subdirectly irreducible.
\end{lem}
\begin{proof}
    Let ${\bf B}$ be $\beta_{\bf A}$-critical, and let $u\approx 
    v\in\Eq({\bf A})$ 
    be an equation of complexity $\beta_{\bf A}(|B|)$ failing on 
    ${\bf B}$ under the assignment $\phi$.  If $\phi(u)$ and $\phi(v)$ can be
   separated by a non-trivial congruence $\theta$, then ${\bf B}/\theta$ belongs to the variety (as $|B|$ was a jump point) but fails $u \approx v$, which is a contradiction. Thus every non-trivial congruence collapses $\phi(u)$ and $\phi(v)$ so that  ${\bf B}$ is subdirectly irreducible.
    \end{proof}
 The $\beta_{\mathscr{V}}$-critical algebras determine the function $\beta_{\mathscr{V}}$ in the sense that if in the definition of $\beta_{\mathscr{V}}$, the choice of ${\bf B}\notin \mathscr{V}$  is restricted to the $\beta_{\mathscr{V}}$-critical algebras, then the function defined coincides with $\beta_{\mathscr{V}}$.  %Then the following lemma is trivial.
%\begin{lem}\label{lem:bound}
%Let $f:\mathbb{R}\to \mathbb{R}$ be a nondecreasing real function such that every $\beta_{\bf A}$-critical algebra of size less than $n$ fails an equation in $\Eq_{f(n)}({\bf A})$.  Then $\beta_{\bf A}$ is bounded above by $f$.
%\end{lem}
%\begin{proof}
%Say that $\beta_{\bf A}(n)=m$.  Thus there is a $\beta_{\bf A}$-critical algebra ${\bf B}$ with fewer than $n$ elements failing some equation of ${\bf A}$ of $m$, but satisfying $\Eq_{m-1}({\bf A})$.  Hence ${\bf B}$ fails some equation satisfied by ${\bf A}$ and of length at most $f(|B|)$.  So $f(n)\geq f(|B|)\geq m=\beta_{\bf A}(n)$, as required.
%\end{proof}
\begin{thm}\label{thm:log3}
Let ${\bf G}$ be a finite  group containing a nonabelian Sylow subgroup.  Then the equational complexity of $\flat({\bf G})$ is not eventually constant but grows within $O(\log_2^3(n))$.
\end{thm}
\begin{proof}
By Lemma \ref{lem:bcritical} we only need to show that there is a constant $c>0$ such that for $n$ sufficiently large, $n$-element subdirectly irreducible algebras outside of $\mathbb{V}(\flat({\bf G}))$ fail an equation in $\Eq(\flat({\bf G}))$ of size at most $c\log_2^3(n)$.

The Oates-Powell Theorem \cite{oatpow} shows that variety of ${\bf G}$ can be axiomatized by a finite system of equations.  As this variety is a quasivariety and as equations are quasi-equations,  Lemma \ref{lem:translate} shows that there is a finitely system $\Sigma$ of equations in the signature $\{\wedge,\cdot,1\}$ for which the subdirectly irreducible models are (up to isomorphism) precisely the algebras $\flat({\bf H})$, where ${\bf H}\in\mathbb{V}({\bf G})$; the same applies in the signature $\{\wedge,\cdot\}$.

Let $m$ be greater than the size of the longest equation in $\Sigma$, and consider a subdirectly irreducible algebra not in $\mathbb{V}(\flat({\bf G}))$, but satisfying all equations of $\flat({\bf G})$ up to size $<m$.  In particular, ${\bf S}$ satisfies $\Sigma$ so is of the form $\flat({\bf H})$ for some finite group ${\bf H}\in\mathbb{V}({\bf G})$.  As $\flat({\bf H})\notin\mathbb{V}(\flat({\bf G}))$,  Lemma \ref{lem:translate} shows that ${\bf H}\notin \mathbb{Q}({\bf G})$ and so Lemma \ref{lem:shortqe} shows that there is a quasi-equation of size $O(\log_2^3(|H|))$ failing on ${\bf H}$ but holding on ${\bf G}$.  Then Lemma~\ref{lem:shorteq} shows that there is an equation of size $O(\log_2^3(|H|))$ failing on ${\bf H}$ and holding on $\flat({\bf G})$.  
\end{proof}

Finally we may complete the proof of the main result.
\begin{proof}[Proof of Theorem \ref{thm:mainClifford}]
Lemma \ref{lem:singlegroup} implies that $\mathbb{V}({\bf C})$ is equal to $\mathbb{V}(\flat({\bf G}))$, where ${\bf G}$ is a group.  Moreover, ${\bf C}$ has a subgroup with a nonabelian Sylow subgroup if and only if~${\bf G}$ has a nonabelian Sylow subgroup.

By \cite{ols}, the group ${\bf G}$ has a finite basis for its quasi-equations if and only if all of its Sylow subgroups are abelian.  Hence, by Lemma \ref{lem:singlegroup}, it follows that $\mathbb{V}({\bf C})=\mathbb{V}(\flat({\bf G}))$ has a finite basis for its equations if and only if ${\bf C}$ has all of its Sylow subgroups abelian.  When ${\bf C}$ has a finite equational basis we obtain $\beta_{\bf C}\in O(1)$.  Otherwise, Theorem \ref{thm:log3} shows that $\beta_{\bf C}$ is unbounded but in $O(\log_2^3(n))$.
\end{proof}

We now observe a crude lower bound, which when combined with Theorem \ref{thm:mainClifford}, sandwiches $\beta_{{\bf C}}(n)$ between functions of growth rate $\log^{\Theta(1)}(n)$.
\begin{thm}
When ${\bf G}$ is a finite group with a nonabelian Sylow subgroup, the function $\beta_{\flat({\bf G})}$ grows in $\Omega(\sqrt[4]{\log(n)})$.
\end{thm}
\begin{proof}
Let ${\bf H}={\bf H}_n$ be the group discussed at the start of Section \ref{sec:group}.  
Then $\overline{\beta}_{\bf G}(|H_n|)>n$.  %It follows from \cite[Lemma~6]{MSzW} (with $f(n):=n$) that if $d(n)$ bounds the size of $\flat({\bf H})$, then $\overline{\beta}_{\bf G}(n)$ eventually dominates $d^{-1}(n-1)$.  
We show that there is a constant $c$ (depending on ${\bf G}$) such that~$2^{cn^{4}}$ bounds $|H|$, from which the claimed lower bound for $\beta_{\flat({\bf G})}(n)$ follows.

Let $m:=(4\binom{n}{2}+4)$.  
The group ${\bf H}$, as constructed by Ol'shanski\u{\i} \cite{ols}, lies in the variety generated by a subgroup $\overline{\bf G}$ of ${\bf G}$ that is a nilpotent group of nilpotency class $2$.  (Note that the subgroup $\overline{{\bf G}}$ is denoted by $H$ in \cite{ols}, while our ${\bf H}$ is denoted by $C$.)  The group~${\bf H}$ is a quotient of the direct product of $s:=2^{|G|}+1$ copies of the $m$-generated free group~${\bf F}$ in the variety generated by $\overline{\bf G}$. 
By Neumann \cite{pneu} and Higman \cite{hig}, the logarithm of $|F|$ (as a function of $m$) grows within $O(m^2)$, so that the same is true of $|F^s|$ and therefore~$|H|$ also.  Thus there is a constant $c$ such that $2^{cn^{4}}$ bounds $|H|$ as required.
\end{proof}

\begin{remark}\label{rem:semiring}
When a finite group ${\bf G}$ is considered in the signature $\{\cdot\}$, then the algebra $\flat({\bf G})$ is an example of an additive idempotent semiring, which also has $O(\log_2^3(n))$ growth equational complexity provided ${\bf G}$ has a nonabelian Sylow subgroup.
\end{remark}
The smallest groups with a nonabelian Sylow subgroup are the $8$-element nonabelian groups.  Thus, the methods in this section produce $9$-element examples of semilattice-ordered Clifford semigroups and semirings with slow growth equational complexity.

\section*{Acknowledgement}
The author wishes to thank the anonymous referee, whose careful reading improved several sections of the article.

%\bibliographystyle{amsplain}

%%%%%%%%%%%%%%%%%%%%%%%%%%%%%%%%%%%%%%%%%%%%%%%%%%%%%%%%%%%%%%%%%%
%%%%%%%%%%%%%%%%%%%%%%%%%%%%%%%%%%%%%%%%%%%%%%%%%%%%%%%%%%%%%%%%%%

\begin{thebibliography}{99}
\bibitem{alm} J. Almeida, Finite Semigroups and Universal Algebra, World Scientific, Singapore 1994.

\bibitem{bgklp} L. Babai, A.J. Goodman, W.M. Kantor, E.M. Luks, P.P. P\'alfy, Short presentations for finite groups, \emph{J. Algebra}, {194} (1997), 79--112.

\bibitem{bir} G. Birkhoff, On the Structure of Abstract Algebras,  \emph{Math. Proc. Cambridge Philos. Soc.} 31 (1935), 433--454.

\bibitem{bursan} S. Burris and V.P. Sankappanavar, \emph{A Course in Universal Algebra}, Graduate Texts in Mathematics, Springer Verlag, 1981.

\bibitem{eilsch} S. Eilenberg and M.-P. Sch\"utzenberger, On pseudovarieties, \emph{Adv. Math.} 19 (1976), 413--418.

\bibitem{eil} S. Eilenberg, \emph{Automata, Languages and Machines} Vol B, Pure and Applied Mathematics, Academic Press, New York, 1976.

\bibitem{gor} V.A. Gorbunov, \emph{Algebraic Theory of Quasivarieties}, Consultants Bureau, New York, 1998.

\bibitem{hig} G. Higman, The orders of relatively free groups, in: \emph{Proc. Internat. Conf. Theory of Groups}, Austral. Natl. Univ. Canberra, 1965, pp.153--165.

\bibitem{hirhod} R. Hirsch and I. Hodkinson, Representability is not decidable for finite relation algebras,  \emph{Trans. Amer. Math. Soc.} 353 (2001), 1403--1425.

\bibitem{how}  J.M. Howie, Fundamentals of Semigroup Theory, 2nd edition. Oxford University Press, New York (1995)


 \bibitem{jac:flat} M. Jackson, Flat algebras and the translation of universal Horn logic to equational logic, \emph{J. Symbolic Logic}, {73} (2008), 90--128.
 
 \bibitem{jac:flexi} M. Jackson, Flexible constraint satisfiability and a problem in semigroup theory, manuscript, 2015 	arXiv:1512.03127.
 
\bibitem{jacmck} M. Jackson and R. McKenzie, Interpreting graph colorability in finite semigroups, \emph{Internat. J. Algebra Comput.} 16 (2006), 119--140.

 \bibitem{jacmcn} M. Jackson and G. F. McNulty, The equational complexity of Lyndon's algebra, \emph{Algebra Univers.} 65 (2011), 243--262.
 
 \bibitem{jacsto:06} M. Jackson and T. Stokes, Identities in the algebra of partial maps, \emph{Internat. J. Algebra Comput.} 16 (2006), 1131--1159.

%\bibitem{khasap} O. Kharlampovich and M. Sapir, Algorithmic problems in varieties, Internat. J. Algebra Comput. 5 (1995), 379--602.

\bibitem{KKP} O. Kl\'{\i}ma, M. Kunc and L. Pol\'ak, Deciding piecewise $k$-testability, manuscript, 2014.

\bibitem{koz04} M. Kozik, \emph{On Some Complexity Problems in Finite Algebras}, Ph.D. thesis, Vanderbilt
University, 2004.


\bibitem{koz09} M. Kozik, A 2EXPTIME complete varietal membership problem, \emph{SIAM J. Comput.} {38} (2009),  2443--2467.


\bibitem{kunver} G. Kun, and V. V\'ertesi, The membership problem in finite flat hypergraph algebras, \emph{Internat. J. Algebra Comput.} 17 (2007), 449--459.


\bibitem{law} M.V. Lawson, \emph{Inverse Semigroups: the theory of partial symmetries}, World Scientific, 1998.

\bibitem{lee1} J. Leech, Inverse monoids with a natural semilattice ordering, \emph{Proc. London Math. Soc.} (3) 70 (1995) 146--182.

\bibitem{lee2} J. Leech, On the foundations of inverse monoids and inverse algebras, \emph{Proc. Edinburgh Math. Soc.} (2) 41 (1998) 1--21.

\bibitem{MSzW}  G.F. McNulty, Z. Sz\'ekely, R. Willard, Equational complexity of the finite algebra
membership problem, \emph{Internat. J. Algebra Comput.} 18 (2008), 1283--1319.

\bibitem{neu} H. Neumann, \emph{Varieties of groups}, Ergebnisse der Mathematik und ihrer Grenzgebiete Volume 37, Springer-Verlag, 1967. 

\bibitem{pneu}  P. Neumann, Some indecomposable varieties of groups, \emph{Q. J. Math.} 14 (1963) 46--50.

\bibitem{ols} A.Ju. Ol$'$\u{s}anski\u{\i}, Conditional identities of finite groups, Sibirsk. Mat. Zh. 15 (1974), 1409--
1413 [Russian; English version in Siberian Math. J. 15 (1975), 1000--1003].

\bibitem{pet} M. Petrich, \emph{Inverse Semigroups}, Wiley, 1984.

\bibitem{oatpow}
S. Oates and M.B. Powell, Identical relations in finite groups, \emph{J. Algebra} {1} (1964), 11--39. 


\bibitem{rhoste} J. Rhodes and B. Steinberg, J.~Rhodes and B.~Steinberg.
\newblock {\em The {$q$}-theory of finite semigroups}.
\newblock Springer Monographs in Mathematics. Springer, New York, 2009.

\bibitem{vol} M.V. Volkov,  Reflexive relations, extensive transformations and piecewise testable languages of a given height, \emph{Internat. J. Algebra Comput.} 14 (2004), 817--827.

\end{thebibliography}
\end{document}